\documentclass[12pt]{article}

\usepackage{amssymb}
\usepackage{amsthm}
\usepackage{amsmath}
\usepackage{graphicx}
\usepackage{fullpage}
\usepackage{color}
 \numberwithin{equation}{section}

\newcounter{mtheorem}

\newtheorem{theorem}{Theorem}[section]

\newtheorem{lemma}[theorem]{Lemma}

\newtheorem{prop}[theorem]{Proposition}

\theoremstyle{definition}

\newtheorem{definition}[theorem]{Definition}

\theoremstyle{remark}

\newtheorem{remark}[theorem]{Remark}

\newcommand{\R}{\mathbb{R}}

\newcommand{\dist}{\operatorname{dist}}

\newcommand{\RD}{\mathbb{R}^D}

\newcommand{\RDD}{\mathbb{R}^D \times \mathbb{R}^D}

\newcommand{\calL}{\mathcal{L}}

\newcommand{\M}{\mathcal{M}}

\newcommand{\calS}{\mathbb{S}}

\newcommand{\lfields}{L^2(\mu)}

\newcommand{\calLone}{\mathcal{L}^1}

\newcommand{\bp}{\begin{proof}[\ensuremath{\mathbf{Proof}}]}
\newcommand{\bs}{\begin{proof}[\ensuremath{\mathbf{Solution}}]}
\newcommand{\ep}{\end{proof}}

\begin{document}

\title{ Moreau-Yosida approximation and convergence of Hamiltonian systems on Wasserstein space }

\author{Hwa Kil Kim\footnote{Courant Institute of Mathematical Sciences, New York University, 251 Mercer Street, New York, NY 10012, USA.
{\it Email address : hwakil@cims.nyu.edu}}}

\maketitle

\begin{abstract}

In this paper, we study the stability property of Hamiltonian systems on the Wasserstein space. Let $H$ be a given Hamiltonian satisfying
certain properties. We regularize $H$ using the Moreau-Yosida approximation and denote it by $H_\tau.$
We show that  solutions of the Hamiltonian system for $H_\tau$ converge to a solution of the Hamiltonian system for $H$ as $\tau$ converges to zero.
 We provide sufficient conditions on $H$ to carry out this process.
\end{abstract}


\section{Introduction}

Let $\mathcal{M}$ be the set of Borel probability measures on $\RD$ with finite second moments equipped with the Wasserstein metric. We study a Hamiltonian type evolution problem in
$\mathcal{M}$ of following form :

\begin{equation}\label{eq1}
   \left \{
  \begin{array}{l}
 \frac{d}{dt}\mu_t + \nabla \cdot(\mathbb{J} v_t\mu_t) = 0, \qquad t \in (0 , T) \\

  v_t \in \partial _-  H(\mu _t)\cap T_{\mu_t}\mathcal{M}, \\
  \end{array}
  \right.
\end{equation}
where the given function $H : \mathcal{M} \rightarrow (-\infty, \infty]$ is referred to as a Hamiltonian. Here
$\mathbb{J}: \mathbb{R}^D \rightarrow \mathbb{R}^D$ is a matrix satisfying $\mathbb{J}v\bot v$
 for all $v\in \mathbb{R}^D$. When $D=2d$ then we can simply set $\mathbb{J}$ to be the $(2d) \times (2d)$ canonical symplectic matrix. Here, $\partial_-H(\mu)$
 denotes the subdifferential of $H$ at $\mu \in \mathcal{M}$ and $T_\mu \mathcal{M}$ is the tangent space at $\mu$ in $\mathcal{M}$
 which will be defined below. There are various reasons for the terminology  is {\it Hamiltonian} type. For example, (\ref{eq1}) is, roughly speaking, a limit of finite dimensional
 Hamiltonian ODE \cite{CGP:article}. Geometric justification was also made in \cite{GKP}.

 The first systematic study addressing evolution problems on $\mathcal{M}$ of the Hamiltonian type was made by Ambrosio and Gangbo
\cite{AG}. They studied the Hamiltonian system for locally subdifferentiable Hamiltonians and proved the existence of a solution. The theory in \cite{AG} covers a large class of systems which have recently generated a lot of interest, including the
Vlasov-Poisson in one space dimension \cite{Brenier}\cite{MZ}, the Vlasov-Monge-Ampere \cite{BL:article}\cite{CGP:article} and the semigeostophic systems
 \cite{BB:article}\cite{Cullen-feldman}\cite{cullen-gangbo}\cite{CGP:article} are casted into the Hamiltonian type formalism.

We are interested in the stability property of Hamiltonian systems in the following sense. Let $H$ be a given Hamiltonian. We ask ourselves whether
there is any {\it regular} approximation $H_\tau$ of $H$ such that solutions of (\ref{eq1}) for $H_\tau$ exist and converge to a solution of
the system (\ref{eq1}) for $H$ as the approximation parameter $\tau$ goes to zero. 

 Since the Wasserstein space is an infinite dimensional metric space, the existence of such an approximation is not a simple question. In this paper, we
 show that the Moreau-Yosida approximation is the one we are looking for. Let $H$ be a Hamiltonian satisfying assumptions (H1) and (H2) whose statements
  will be given later. We first regularize the Hamiltonian $H$ to obtain $H_{\tau}$ defined by
 $$H_{\tau}(\mu) = \inf_{\nu\in \mathcal{M}} \{1/2\tau W(\mu,\nu)^2 + H(\nu)   \}.$$
 The new functional $H_\tau$ is $1/\tau-$concave even if $H$ is not. Next, we apply the algorithm developed in \cite{AG} to solve
\begin{equation}
   \left \{
  \begin{array}{l}
 \frac{d}{dt}\mu_t^\tau + \nabla \cdot(\mathbb{J} v_t^\tau\mu_t^\tau) = 0, \qquad t \in (0 , T) \\

  v_t^\tau \in \partial^+  H_\tau(\mu _t^\tau)\cap T_{\mu_t^\tau}\mathcal{M}, \\
  \end{array}
  \right.
\end{equation}
where $\partial^ + H_\tau(\mu_t^\tau)$ is the superdifferential of $H_\tau$ at $\mu_t^\tau$ in the sense of \cite{ags:book}.
Finally, we show, for any sequence $\tau_n$ converging to zero, $\mu^{\tau_n}$
(up to subsequence) converges to $\mu$ which is a solution of (\ref{eq1}).

Our assumptions on the Hamiltonian $H$ allow $H$ to be no locally subdifferentiable.
Hence, our stability result allow us to construct solutions to the system (\ref{eq1}) for Hamiltonians which are not everywhere subdifferentiable around
the initial measure. This is not the case in \cite{AG}. At the end of this paper, we will discuss more about how the Moreau-Yosida approximation scheme
is useful in the study of non locally subdifferentiable Hamiltonians.

We briefly summarize the contents of each section. Section 2 is a preliminary on the Wasserstein space $\M.$ In section 3, we give an introduction to
the  Moreau-Yosida approximation of functionals defined on $\mathcal{M}$ and investigate some properties of it. The main feature in this section is
 Lemma \ref{lemma:stability} which is the key ingredient to prove Theorem \ref{thm-hf}.
 In section 4, we prove our main stability result Theorem \ref{thm-hf} under assumptions (H1) and (H2) on the Hamiltonian $H.$ We show Hamiltonians
 considered in \cite{AG} satisfy (H1) and (H2), and so corresponding Hamiltonian systems are stable w.r.t Moreau-Yosida approximation.
Let us close this introduction by fixing notations and terminologies.


\subsection{Notation and Terminology}

- $\mathcal{P}(\RD)=\{\mu| \mu$ is a \textit{Borel probability measure} on $\mathbb{R}^{D}\}$\\
- Let $\M$ be the subspace of $\mathcal{P}(\mathbb{R}^D)$ with bounded second moment, \textit{i.e.}
$$\M:=\Bigl \{\mu \in\mathcal{P}(\RD): \mu\geq 0, \int_{\RD}d\mu=1, \int_{\RD}|x|^2\,d\mu<\infty \Bigr \}.$$
- Let $\mu\in \mathcal{P}(\RD)$ and let $f: \RD \rightarrow \R^k$ be a \textit{ Borel map}. Then $\nu:=f_{\#}\mu$
is a \textit{Borel measure} on $\R^k$ characterized by $\nu[B] = \mu[f^{-1}(B)]$ for all
\textit{Borel sets} $B\subset\R^k$. In this case, we say $f$ {\it pushes} $\mu$ {\it forward to} $\nu$. \\
- $C_c^\infty(\RD) $ is the collection of all infinitely differentiable  functions with compact
support.\\
- We denote $C_b(\RD) $ the collection of all continuous and bounded  functions.\\
- Let $\mu_n, \mu\in \mathcal{P}(\RD),$ we define $\mu_n$ converges {\it narrowly} to $\mu$ if
$$\int_{\RD} f(x) d\mu_n(x) \longrightarrow \int_{\RD}f(x) d\mu(x) \quad {\rm as} \quad n\rightarrow \infty,$$
for any $f \in C_b(\RD),$ {\it i.e.} $\mu_n$ weak* converges to $\mu.$  \\
- $Id : \RD\rightarrow \RD$ is the identity map,{\it i.e.} $Id(x) =x$ for all $x\in \RD.$\\
- $\pi^i, \pi^{i,j} : \R^{nD} \rightarrow \RD,\RDD$ are the standard projections, {\it i.e.}
$$\pi^i(x_1,x_2,\cdots,x_n)=x_i \quad \mathrm{and} \quad \pi^{i,j}(x_1,x_2,\cdots,x_n)=(x_i,x_j).$$
- Let $\mu\in \mathcal{P}(\RD)$ and let $f:\RD\rightarrow \R^k.$ We denote
the $L^2$ norm of $f$ by $||f||_\mu,$  {\it i.e.}
$$     ||f||^2_\mu:=||f||^2_{L^2(\mu)}= \int_{\RD} |f(x)|^2 d\mu(x) .   $$
- Let $\mu\in \mathcal{P}(\RD)$, we denote the support of $\mu$ by $supp(\mu).$\\
- Let $r >0$ and $x\in \RD$ then $B_x(r)$ denotes the open ball in $\RD$ of center $x$ and radius $r$.\\
- Let $x,y \in \RD$, we denote the inner product of $x$ and $y$ by $<x,y>$.\\

\section{Wasserstein space}
Recall that  $\M$ is the subspace of $\mathcal{P}(\mathbb{R}^D)$ with bounded second moment. In this section, we show that $\mathcal{M}$ has a metric
structure and we introduce a differentiable structure in $\mathcal{M}$. We refer to \cite{ags:book} and \cite{Villani:book} for further details.
\subsection{Wasserstein distance}

\begin{definition}\label{def:wasserstein}

Let $\mu,\,\nu\in\M.$ Consider
\begin{equation} \label{eq:wasserstein}
W_2(\mu,\nu):= \Bigl (\inf_{\gamma\in\Gamma(\mu,\nu)}\int_{\RD\times\RD} |x-y|^2 d\gamma(x,y)\Bigr)^{1/2}.
\end{equation}
Here, $\Gamma(\mu, \nu)$ denotes the set of Borel measures $\gamma$ on $\RD \times \RD$ which have
$\mu$ and $\nu$ as \textit{marginals}, \textit{i.e.} satisfying
$\pi^1_{\#}(\gamma)=\mu$ and $\pi^2_{\#}(\gamma)=\nu$.

Equation (\ref{eq:wasserstein}) defines a metric on $\M$ which is called the \textit{Wasserstein distance}.
It is known that the infimum in the right hand side of equation (\ref{eq:wasserstein}) is always achieved.
We will denote by $\Gamma_o(\mu,\nu)$ the set of $\gamma$ which achieve the minimum in (\ref{eq:wasserstein}).

\end{definition}

\begin{definition}\label{barycentricprojection}
Let $\mu,\nu\in \M$ and  $\gamma\in \Gamma_o(\mu,\nu).$
The \textit{ barycentric projection} $\bar{\gamma}_\mu^\nu : \RD \rightarrow \RD$ of $\gamma$ with respect
to the first marginal $\mu$ is characterized by
\begin{equation}\label{barycentric:firstmarginal}
\int_{\RD} \psi(x)\bar{\gamma}_{\mu}^{\nu}(x)d\mu(x) = \int_{\R^{2D}} \psi(x)yd\gamma(x,y) \quad
\forall \psi \in C_b.
\end{equation}

\noindent
Similarly, the \textit{ barycentric projection} $\bar{\gamma}_\nu^\mu : \RD \rightarrow \RD$ of $\gamma$ with respect
to the second marginal $\nu$ is defined by
\begin{equation}\label{barycentric:secondmarginal}
\int_{\RD} \psi(y)\bar{\gamma}^{\mu}_{\nu}(y)d\nu(y) = \int_{\R^{2D}} \psi(y)xd\gamma(x,y) \quad
\forall \psi \in C_b.
\end{equation}
\end{definition}

%

\subsection{Differential structure on $\mathcal{M}$}

\begin{definition} \label{def:tangentspaces}
Given $\mu \in \M$, let $T_\mu \M$ be the \textit{tangent space} of $\M$ at $\mu$ defined as the closure of $\nabla C_c^\infty$ in $\lfields$, {\it i.e.}
\begin{equation*}
T_\mu\mathcal{M}:= \overline{\{\nabla \varphi : \varphi\in C_c^\infty(\RD) \}}^{L^2(\mu)}.
\end{equation*}

For any $\mu \in \mathcal{M}$, there is an orthogonal decomposition
 \begin{equation}\label{orthogonal:decomposition}
L^2(\mu)=T_\mu \mathcal{M} \oplus [T_\mu \mathcal{M}]^\bot,
 \end{equation}
 where $[T_{\mu}\mathcal{M}]^\bot:= \{w \in L^2(\mu): \nabla \cdot (w\mu) = 0    \} .$  We will denote by $\pi_\mu : L^2(\mu) \rightarrow
 T_\mu \mathcal{M}$ the corresponding orthogonal projection.
\end{definition}

As shown in \cite{ags:book}, the tangent space enjoys many useful properties in analytic and geometric point of view. Here, we
 recall one of them which is related to absolutely continuous curves in $\mathcal{M}.$ Let us first give the definition of absolutely continuous curves in metric spaces.

%
%
%

\begin{definition}\label{deac2curves}

Let $(\calS, \dist)$ be a metric space. A curve $t \in (a,b)
\mapsto \sigma_t \in \calS$ is \textit{2--absolutely continuous}
if there exists $\beta \in L^2(a,b)$ such that
\begin{equation}
\dist(\sigma_t, \sigma_s) \leq \int_s^t \beta(\tau)d\tau,
\end{equation}
\noindent
for all $a <s<t<b.$ We then write $\sigma \in AC_2(a,b; \calS).$
For such curves the limit $|\sigma'|(t):=\lim_{s \rightarrow t}
\dist(\sigma_{t}, \sigma_s)/|t-s| $ exists for $\calL^1$--almost
every $t \in (a,b)$. We call this limit the \textit{metric
derivative} of $\sigma$ at $t.$ It satisfies $|\sigma'| \leq
\beta$ $\calL^1$--almost everywhere.

\end{definition}
\noindent
We now recall Theorem 8.3.1 in \cite{ags:book}, which says that the tangent space provides a canonical velocity field for the
absolutely continuous curves in $\M.$
%
%

\begin{prop}\label{prac2curves}

If $\mu \in AC_2(a,b; \M)$ then there exists a Borel map $v:
(a,b)\times\RD\rightarrow\RD$ such that
 $v_t \in L^2(\mu_t)$ for $\calLone$--almost every $t \in (a,b)$ and
 $$\frac{\partial\,\mu_t}{\partial t}+\nabla \cdot (v_t \mu_t)=0.$$
 We call
$v$ a \textit{velocity} for $\mu.$ If $w$ is another velocity
for $\mu$ then $\pi_{\mu_t}(v_t)=\pi_{\mu_t}(w_t)$ for $\calLone$--almost every $t \in
(a,b),$ where $\pi_{\mu_t}$ is defined in Definition \ref{def:tangentspaces}. Moreover, one can choose $v$ such that $ v_t  \in  T_{\mu_t}\M$
and $||v_t||_{\mu_t}=|\mu'|(t)$ for $\calL^1$--almost every
$t \in (a,b)$. In that case, for $\calL^1$--almost every $t \in
(a,b)$, $v_t$ is uniquely determined. We refer to $ v$ as the \textit{velocity of minimal
norm}, since if $w$ is any other velocity associated to $\mu$
then $||v_t||_{\mu_t} \leq ||w_t||_{\mu_t}$ for
$\calL^1$--almost every $t \in (a,b)$ and so
$\dist(\mu_t, \mu_s) \leq \int_s^t ||v_\tau||_{\mu_\tau}d\tau \leq
\int_s^t || w_\tau||_{\mu_\tau}d\tau$
for all $a <s<t<b.$

\end{prop}

%
%
%

Following \cite{AG}, we give a notion of a differential and a definition of convex functions on $\M$.

\begin{definition} \label{def:Fdifferentiable}
Let $H:\M\rightarrow (-\infty,\infty]$ be a proper function on $\M$, {\it i.e.} the effective domain of $H$ defined by
$D(H):=\{\mu \in \mathcal{M} : H(\mu) < \infty \}$ is not empty. We say that
$\xi \in L^2(\mu)$ belongs to the \textit{subdifferential}
$\partial_{-} H(\mu)$ if
$$ H(\nu) \geq H(\mu) + \sup_{\gamma \in \Gamma_o(\mu,\nu)}\int_{\RD\times \RD} \langle\xi(x),y-x\rangle d\gamma(x,y) + o(W_2(\mu,\nu)) ,$$
as $\nu \rightarrow \mu.$ We denote the domain of subdifferential by $D(\partial_-H):=\{\mu: \partial_-H(\mu)\neq \emptyset   \}$.\\
 If $-\xi \in \partial_{-} (-H)(\mu)$ then
we say that $\xi$ belongs to the \textit{superdifferential}
$\partial^+ H(\mu)$.
 \end{definition}

 \begin{remark}
If $ \partial_{-} H(\mu) \cap \partial^+ H(\mu)\neq \emptyset$
then we say that $H$ is \textit{differentiable} at
$\mu.$ In this case, there is a unique vector in $\partial_{-} H(\mu) \cap
\partial^+ H(\mu)\cap T_{\mu}\M$ and we define the \textit{gradient vector} $\nabla_\mu H$ by the unique vector.

\end{remark}



%

\begin{definition}\label{lambdaconvexity}
Let $H: \M \rightarrow (-\infty,\infty]$ be proper and let $\lambda \in \R.$ We say that $H$ is \textit{$\lambda$- convex} if for every
$\mu_0,\mu_1\in \M$ and every optimal transport plan $\gamma\in \Gamma_o(\mu_0,\mu_1)$ we have
 \begin{equation}
H(\mu_t)\leq (1-t)H(\mu_0) + t H(\mu_1) -\frac{\lambda}{2}t(1-t)W_2^2(\mu_0,\mu_1) \quad \forall t\in[0,1],
\end{equation}
where $\mu_t = ((1-t)\pi^1 + t\pi^2)_{\#}\gamma$. If $-H$ is $(-\lambda)$-convex then $H$ is called \textit{$\lambda$- concave}.
\end{definition}

%
%


\section{Moreau-Yosida approximation}
In this section, we introduce the Moreau-Yosida approximation of functionals on $\mathcal{M}.$ 
\begin{definition}
 Let $H : \M \rightarrow (-\infty, \infty]$ be a proper and coercive functional. For $\tau >0$, the {\it Moreau-Yosida
approximation} $H_{\tau}$ of $H$ is defined as
\begin{equation}\label{MY-eq1}
H_{\tau}(\mu)  = \inf_{\nu\in \M}\Bigl \{\frac{1}{2\tau} W_2^2(\mu,\nu) + H(\nu)\Bigr \},
\end{equation}
Here, $H$ is coercive means that there exist $\tau_*>0$ and $\mu_*\in\M$ so that $H_{\tau_*}(\mu_*)> -\infty.$ We also set
 \begin{equation}
\it{J}_{\tau}[\mu] := \Bigl \{ \mu_{\tau} : H_\tau(\mu)= \frac{1}{2\tau} W_2^2(\mu,\mu_\tau) + H(\mu_\tau) \Bigr \}.
\end{equation}

\end{definition}


\begin{lemma}\label{lemma:concave}
Let $H : \M \rightarrow (-\infty, \infty]$ be a proper and coercive functional, and $H_{\tau}$ be the
Moreau-Yosida approximation of $H$. Then $H_{\tau}$ is
$\frac{1}{\tau} $ - concave.
\end{lemma}

\begin{proof}
Let $\nu\in \M$ be fixed, then it is well known that $\mu\mapsto \frac{1}{2}W_2^2(\mu,\nu)$ is a $1$-concave function on $\mathcal{M}$. This implies

$$ \mu\rightarrow H_\tau(\mu)= \inf_{\nu\in\mathcal{M}}\Bigl \{\frac{1}{2\tau}W_2^2(\mu,\nu) + H(\nu)  \Bigr \},$$
is $\frac{1}{\tau}-$concave since it is an infimum of $\frac{1}{\tau}-$concave functionals.
\end{proof}

We now introduce two Lemmas which give relations between the subdifferential of $H$ and the superdifferential of $H_\tau$.
They play the key role in the convergence of Hamiltonian systems.

\begin{lemma}\label{MYSP}
Let $H : \M \rightarrow (-\infty, \infty]$ be a proper functional and $H_\tau$ be the Moreau-Yosida approximation of $H$. For $\mu_o\in\M$ given,
if $\nu_o \in J_{\tau}[\mu_o]$ then
$H_\tau$ is superdifferentiable at $\mu_o$ and $H$ is subdifferentiable at $\nu_o,$ $i.e.$ $\mu_o\in D(\partial^+H_\tau)$ and $\nu_o\in D(\partial_-H)$.
Furthermore, for any $\gamma \in \Gamma_o(\mu_o,\nu_o)$, we have
\begin{equation}\label{subdiff:superdiff:lemma:MYSP}
\frac{Id-\bar{\gamma}_{\mu_o}^{\nu_o}}{\tau}\in \partial^+H_\tau(\mu_o) \cap T_{\mu_o}\mathcal{M} \quad ,
\quad \frac{\bar{\gamma}_{\nu_o}^{\mu_o}-Id}{\tau}\in \partial_-H(\nu_o) \cap T_{\nu_o}\mathcal{M}
\end{equation}
where $\bar{\gamma}_{\mu_o}^{\nu_o}$($\bar{\gamma}_{\nu_o}^{\mu_o}$) is the barycentric
projection of $\gamma$ with respect to the first(respectively, second) marginal as in Definition \ref{barycentricprojection}.
\end{lemma}

\begin{proof}
From the definition of  $\nu_o \in J_{\tau}[\mu_o]$, we have
\begin{equation}\label{eq4:lemma:MYSP}
H_{\tau}(\mu) - H_{\tau}(\mu_o) \leq \frac{1}{2\tau}W_2^2(\mu,\nu_o) - \frac{1}{2\tau}W_2^2(\mu_o,\nu_o)\quad \forall \mu\in\mathcal{M},
\end{equation}

\noindent
For a fixed $\mu,$ we choose $\eta \in \Gamma_o(\mu_o,\mu).$ Let $\eta =\int_{\RD}\eta_{x}d\mu_o(x)$ and $ \gamma=\int_{\RD}\gamma_{x}d\mu_o(x)$ be the
disintegration of $\eta$ and $\gamma$ w.r.t $\mu_o.$ Define $\mathbf{u}_1\in \mathcal{P}(\R^{3D})$ to be such that the disintegration of $\mathbf{u}_1$
w.r.t $\mu_o$ is
\begin{equation}\label{eq3:lemma:MYSP}
\int_{\RD}\eta_{x}\times\gamma_{x}d\mu_o(x).
\end{equation}

\noindent
We combine (\ref{eq4:lemma:MYSP}) and
(\ref{eq3:lemma:MYSP}) to get
\begin{align}\label{eq5:lemma:MYSP}\nonumber
H_{\tau}(\mu) - H_{\tau}(\mu_o) &\leq \frac{1}{2\tau}W_2^2(\mu,\nu_o) - \frac{1}{2\tau}W_2^2(\mu_o,\nu_o)\\\nonumber
  &\leq \frac{1}{\tau}\int_{\R^{3D}} \frac{|y-z|^2}{2}-\frac{|x-z|^2}{2}d\mathbf{u}_1(x,y,z)\\\nonumber
  &= \frac{1}{\tau}\int_{\R^{3D}} \langle x-z,y-x \rangle + \frac{|y-x|^2}{2}d\mathbf{u}_1(x,y,z)\\
  & =\int_{\R^{2D}} \langle \frac{x- \bar{\gamma}_{\mu_o}^{\nu_o}(x)}{\tau} ,y-x \rangle
d\eta(x,y) + \frac{1}{2\tau}W_2^2(\mu_o,\mu)
\end{align}

\noindent
which gives
$$\frac{Id - \bar{\gamma}_{\mu_o}^{\nu_o}}{\tau}  \in \partial^+ H_\tau(\mu_o).$$
 Furthermore, it is well known that
$Id - \bar{\gamma}_{\mu_o}^{\nu_o} \in T_{\mu_o}\mathcal{M}$(Proposition 4.3 of \cite{AG}). This concludes the first inclusion of
(\ref{subdiff:superdiff:lemma:MYSP}).

To prove the second, we again exploit $\nu_o \in J_{\tau}[\mu_o]$ to get

$$\frac{1}{2\tau}W_2^2(\mu_o,\nu_o) + H(\nu_o)\leq \frac{1}{2\tau}W_2^2(\mu_o,\nu)+ H(\nu) \quad \forall \nu\in\mathcal{M}. $$

\noindent
For a fixed $\nu,$ let $\tilde{\eta} \in \Gamma_o(\nu_o,\nu)$ and define $\mathbf{u}_2 \in \mathcal{P}(\R^{3D})$ to be such that whose disintegration w.r.t
$\nu_o$ is
\begin{equation*}
\int_{\RD}\gamma_{x}\times\tilde{\eta}_{x}d\nu_o(x),
\end{equation*}

\noindent
where, $\tilde{\eta} =\int_{\RD}\tilde{\eta}_{x}d\nu_o(x)$ and $ \gamma=\int_{\RD}\gamma_{x}d\nu_o(x)$ are disintegrations of $\tilde{\eta}$ and
$\gamma$ w.r.t $\nu_o.$ Computations as in (\ref{eq5:lemma:MYSP}) give
\begin{align*}
H_{\tau}(\nu) - H_{\tau}(\nu_o) \geq \int_{\R^{2D}} \langle \frac{\bar{\gamma}_{\nu_o}^{\mu_o}(x)-x}{\tau} ,y-x \rangle
d\tilde{\eta}(x,y) + \frac{1}{2\tau}W_2^2(\nu_o,\nu).
\end{align*}
We conclude the second inclusion in (\ref{subdiff:superdiff:lemma:MYSP}) using the same argument as above.
\end{proof}


 \begin{lemma}\label{lemma:stability}
Let $H:\mathcal{M}\rightarrow (-\infty,\infty]$ be a proper functional and let $H_\tau$ be the Moreau-Yosida approximation of $H$. Given a sequence
of measures $\mu_n$ and $\nu_n$ be such that $\nu_n \in  J_{\tau_n} [\mu_n].$ Furthermore, suppose there is a constant $C$ satisfying
\begin{equation}\label{eq6:lemma:stability}
W_2(\mu_n,\nu_n)\leq C\tau_n,
\end{equation}
for all $n$. If  $\mu_n$ converges narrowly to $\mu$ as $\tau_n\rightarrow 0,$ then $\nu_n$ also converges narrowly to $\mu$. Furthermore, we have
\begin{equation}\label{eq2:lemma:stability}
 \bigcap_{m=1}^{\infty}\bar{co} \bigl (\{ \frac{Id-\bar{\gamma}_{\mu_{\tau_n}}^{\nu_{\tau_n}}}{\tau_n}\mu_{\tau_n}:n\geq m  \} \bigr )=
\bigcap_{m=1}^{\infty}\bar{co} \bigl (\{ \frac{\bar{\gamma}_{\nu_{\tau_n}}^{\mu_{\tau_n}}-Id}{\tau_n}\nu_{\tau_n}:n\geq m  \} \bigr ).
\end{equation}
 Here $\bar{co}$ denotes the closed convex hull with
respect to weak*-topology.
 \end{lemma}

 \begin{proof}
Narrow convergence of $\nu_n$ to $\mu$ is trivial from the assumption (\ref{eq6:lemma:stability}) and the narrow convergence of $\mu_n$ to $\mu$ as $\tau_n \rightarrow 0.$

To prove (\ref{eq2:lemma:stability}), let us fix $\psi \in C_c^{\infty}(\mathbb{R}^D;\RD).$ Then we have
\begin{align}\label{EQ-2}\nonumber
     \int_{\RD}\langle\psi(y) ,\frac{\bar{\gamma}_{\nu_{\tau_n}}^{\mu_{\tau_n}}-Id}{\tau_n}(y)\rangle d\nu_{\tau_n}(y) &= \int_{\mathbb{R}^{2D}}
     \langle\psi(y) , \frac{x-y}{\tau_n}\rangle d\gamma_{\tau_n}(x,y)\\ \nonumber
&= \int_{\mathbb{R}^{2D}}\langle \psi(x) + \nabla \psi(\xi_{x,y})\cdot(y-x) ,  \frac{x-y}{\tau_n}\rangle
d\gamma_{\tau_n}(x,y)\\\nonumber
& = \int_{\mathbb{R}^{D}}\langle \psi(x) ,  \frac{Id-\bar{\gamma}_{\mu_{\tau_n}}^{\nu_{\tau_n}}}{\tau_n}(x)\rangle d\mu_n(x)\\
& \qquad + \int_{\mathbb{R}^{2D}}\langle \nabla \psi(\xi_{x,y})\cdot(y-x) ,  \frac{x-y}{\tau_n}\rangle d\gamma_{\tau_n}(x,y)
 \end{align}
 where $\gamma_{\tau_n} \in \Gamma_o(\mu_{\tau_n},\nu_{\tau_n}) $ and $\xi_{x,y}$ is a point on the line segment connecting $x$ and $y.$
Since $\psi \in C_c^{\infty}(\mathbb{R}^D;\RD)$, we have
\begin{align}\label{EQ-3}\nonumber
\left |\int_{\mathbb{R}^{2D}}\langle \nabla \psi(\xi_{x,y})\cdot(y-x) ,  \frac{x-y}{{\tau_n}}\rangle
d\gamma_{{\tau_n}}(x,y) \right | &\leq \frac{||\nabla \psi ||_{\infty}}{\tau_n}W_2^2(\mu_{\tau_n},\nu_{\tau_n})\\
 &\leq ||\nabla \psi ||_{\infty}C^2{\tau_n}.
 \end{align}

\noindent
 We combine  equations (\ref{EQ-2}) and (\ref{EQ-3}) to get
\begin{equation*}
\lim_{n\rightarrow \infty}\int_{\RD}\langle\psi(y) ,\frac{\bar{\gamma}_{\nu_{\tau_n}}^{\mu_{\tau_n}}-Id}{\tau_n}(y)\rangle d\nu_{\tau_n}(y)
= \lim_{n\rightarrow \infty} \int_{\mathbb{R}^{D}}\langle\psi(x) ,  \frac{Id-\bar{\gamma}_{\mu_{\tau_n}}^{\nu_{\tau_n}}}{\tau_n}(x)\rangle d\mu_n(x),
\end{equation*}
which concludes  (\ref{eq2:lemma:stability}).
 \end{proof}

%
%

 \section{Convergence of Hamiltonian systems w.r.t Moreau-Yosida approximation}
Now we are ready to state our main result on the stability of Hamiltonian systems.
More specifically, solutions of the approximated Hamiltonians converge to a solution of the original
Hamiltonian system. Let us first be clear about the meaning of solution.

\begin{definition}
 Let $H : \M \rightarrow (-\infty,\infty]$ be a proper and lower semicontinuous function. We say
    that an $2-$absolutely continuous curve $\mu_t :[0,T]\rightarrow D(H)$ is a
    {\it solution of the Hamiltonian system} starting from
    $\bar{\mu}\in \M$, if there exist a vector field $v_t\in L^2(\mu_t)$ with $||v_t||_{\mu_t}\in L^1(0,T)$,
    such that
    \begin{align}\label{eq-1}
   \left \{
  \begin{array}{l}
 \frac{d}{dt}\mu_t + \nabla \cdot(\mathbb{J} v_t\mu_t) = 0,\quad \mu_0=\bar{\mu} \qquad t \in (0 , T) \\
 v_t \in \partial _-  H(\mu _t)\cap T_{\mu_t}\M  \quad a.e\quad t\in (0,T). \\
  \end{array}
  \right.
\end{align}
\end{definition}
\noindent
Equation (\ref{eq-1}) should be understood in the sense of a distribution:
For any $\eta\in C_c^{\infty}(0,T)$ and $\zeta\in C_c^{\infty}(\RD)$, we have
$$\int_0^T\int_{\RD}\eta'(t)\zeta(x) + \eta(t)\langle\nabla \zeta(x):\mathbb{J}v_t(x)\rangle d\mu_t(x)dt=0.$$


\bigskip
\medskip
To ensure the stability of Hamiltonian systems, we require two assumptions on the Hamiltonian.

\noindent
{\it {\bf(H1)} There exist constants $C_o\in [0,\infty), R_o\in(0,\infty]$ such that if $W_2(\mu,\bar{\mu}) < R_o$ then $\mu \in D(H)$ and
, for each $\mu$, there exists a unique $\nu \in J_\tau[\mu]$ satisfying
 \begin{itemize}
 \item
 $\mu\mapsto \nu\in J_\tau[\mu]$ is continuous w.r.t the topology induced by the Wasserstein distance and
 \begin{equation}\label{Condition:H1}
 \frac{W_2(\mu,\nu)}{\tau} \leq C_o.
\end{equation}

\item
 There exists a constant $k>0$ such that
 \begin{equation}\label{Condition-2:H1}
 If \quad supp(\mu) \subset B_0(r)\quad then \quad supp(\nu) \subset B_0(kr),
 \end{equation}
for all $r>0$ and $\mu.$ Recall, $B_0(r)$ is the ball around the origin with radius $r$ in $\RD.$
\end{itemize}}

\noindent
{\bf (H2)} If $\mu_n\in D(\partial_-H)$ and $\mu_n$ converges narrowly to $ \mu$, then $\mu\in D(\partial_-H)$ and we have
\begin{equation}\label{H2:AG}
\bigcap_{m=1}^\infty \bar{co}(\{ w_n\mu_n : w_n\in \partial_- H(\mu_n) \cap T_{\mu_n} \mathcal{M}, n\geq m \}) \subset \{w\mu :
w\in \partial_- H(\mu) \cap T_\mu \mathcal{M} \}.
 \end{equation}

\begin{remark}\label{Remark:H1}
\begin{enumerate}
   \item
   Notice that our Hamiltonian $H$ does not need to be subdifferentiable everywhere in a neighborhood of $\bar{\mu}.$
   We only assume that $D(\partial_-H)$ is closed in the weak* topology and (\ref{H2:AG}) holds.
  \item
   Concerning (H1), suppose $H$ satisfies the following convexity condition for some $\lambda\in \R$:

\noindent
For all $\mu,\nu_0$ and $\nu_1$ in $D(H),$ there exists a curve $\sigma:[0,1]\rightarrow \M$ such that
$\sigma_0=\nu_0,\sigma_1=\nu_1$ and
\begin{equation}\label{eq1:Remark:H1}
\mathcal{H}(\tau,\mu;\sigma_t) \leq (1-t)\mathcal{H}(\tau,\mu;\nu_0) + t\mathcal{H}(\tau,\mu;\nu_1) -\frac{1+\lambda\tau}{2\tau}t(1-t)W_2^2(\nu_o,\nu_1),
\end{equation}
for all $t\in[0,1]$ and $0<\tau<\frac{1}{\lambda^-}$. Here, $\mathcal{H}(\tau,\mu;\nu):=\frac{1}{2\tau}W_2^2(\nu,\mu)+H(\nu).$

Then, Theorem 4.1.2 in \cite{ags:book} says that
if $\mu\in D(H)$ and $\lambda\tau >-1$ then there exists a unique $\mu_\tau\in J_\tau[\mu]$ and the
map $\mu\in D(H)\mapsto \mu_\tau \in J_\tau[\mu]$ is continuous.
 \end{enumerate}
%
%
%
%
\end{remark}

\subsection {Existence of solutions of the regularized Hamiltonian systems}

\begin{lemma}\label{MY-H2}
Let $H:\mathcal{M}\rightarrow (-\infty,\infty]$ be proper and lower semicontinuous, and satisfy (H1).
Let $\mu_n$ be a sequence satisfying $W_2(\mu_n,\bar{\mu})< R_o$ and $\nu_n \in \it{J}_{\tau}[\mu_n].$ If $\mu_n$ converges to $\mu$ in the Wasserstein
topology, then $\nu_n$ also converges to $ \nu\in J_\tau[\mu]$ in the same topology. Furthermore, we have
\begin{equation}\label{eq2:lemma:MY-H2}
\mathcal{K}_o:=\bigcap _{m=1}^{\infty} \overline{co} \bigl ( \bigl \{ \frac{Id-
\bar{\gamma}_{\mu_n}^{\nu_n}}{\tau} \mu_n: n \geq m\bigr \} \bigr ) \subset \bigl \{ \frac{Id-\bar{\gamma}_{\mu}^{\nu}}{\tau} \mu \bigr \},
\end{equation}

\noindent
where $\gamma_n \in \Gamma_o(\mu_n,\nu_n)$ and $\gamma \in \Gamma_o(\mu,\nu).$

\end{lemma}

\begin{proof}
By (H1), there exists a $\nu\in J_\tau[\mu]$ such that $W_2(\nu_n,\nu)\rightarrow 0.$

Next, to prove (\ref{eq2:lemma:MY-H2}), let us assume ${\bf {u}} \in \mathcal{K}_o.$ Then, there exists a sequence $\{\Lambda_m  \}_{m=1}^\infty$
such that
\begin{equation*}
\Lambda_m  =  \sum_{i=m}^{l_m}\lambda_i^m\frac{Id-\bar{\gamma}_{\mu_i}^{\nu_i}}{\tau}\mu_i, \quad \sum_{i=m}^{l_m}\lambda_i^m =1
, \quad 0\leq \lambda_i^m\leq 1,\quad m\leq l_m\in \mathbb{N}
\end{equation*}
and $\Lambda_m $ weak* converges to $\bf{u}$. For any $F\in C_c(\RD;\RD)$,  we have
\begin{align}\label{eq3:lemma:MY-H2}\nonumber
 \int_{\RD} F\cdot d{\bf {u}}
 &= \lim_{m\rightarrow \infty} \sum_{i=m}^{l_m} \lambda_i^m \int_{\RD} \langle F(x), \frac{x-\bar{\gamma}_{\mu_i}^{\nu_i}(x)}{\tau}\rangle d\mu_i(x)\\
&= \lim_{m\rightarrow \infty} \sum_{i=m}^{l_m} \lambda_i^m \int_{\RD} \langle F(x), \frac{x-y}{\tau}\rangle d\gamma_i(x,y),
\end{align}
where $\gamma_i \in \Gamma_o(\mu_i,\nu_i).$ Since $W_2(\mu_n,\nu),W_2(\nu_n,\nu)\rightarrow 0$ as $n\rightarrow \infty,$ there exists
$\gamma\in \Gamma_0(\mu,\nu)$ so that
\begin{equation}\label{eq4:lemma:MY-H2}
\lim_{i\rightarrow \infty}W_2(\gamma_i,\gamma)=0.
\end{equation}
We combine (\ref{eq3:lemma:MY-H2}) and (\ref{eq4:lemma:MY-H2}), to get
\begin{align}\label{eq1:lemma:MY-H2}
 \int_{\RD} F\cdot d{\bf {u}}
 =   \int_{\mathbb{R}^{2D}} \langle F(x), \frac{x-y}{\tau}\rangle d{\gamma}(x,y),
\end{align}
which proves (\ref{eq2:lemma:MY-H2}).
\end{proof}

Now we generate a solution of the Hamiltonian system for $H_\tau.$ The proof of the following theorem is based on Theorem 7.4 in \cite{AG}.
\begin{theorem}\label{thm-super}
Let $H :
\M \rightarrow (-\infty, \infty]$ be a proper and lower semicontinuous functional satisfying the assumption (H1). Let $C_o$ and $R_o$ be constants in
(H1), and set $T=\frac{R_o}{C_o}.$

If $\bar{\mu} \in \M$ has bounded support, then for each $\tau>0,$ there exists a solution of the following Hamiltonian system starting from $\bar{\mu}$
\begin{equation}\label{eq-5}
   \left \{
  \begin{array}{l}
  \frac{d}{dt}\mu_t^{\tau} + \nabla \cdot(\mathbb{J} v_t^{\tau}\mu_t^{\tau}) = 0, \quad \mu_0^\tau=\bar{\mu}\quad t \in (0 , T) \\
  \\
  v_t^{\tau}= \frac{Id-\bar{\gamma}_{\mu_t^\tau}^{\nu_t^\tau}}{\tau} \in \partial^{+}  H_{\tau}(\mu _t^{\tau})\cap T_{\mu_t^{\tau}}M,
  \qquad a.e \quad t\in(0,T),\\
  \end{array}
  \right.
\end{equation}
where $\nu_t^\tau \in J_\tau[\mu_t^\tau].$ Furthermore, $t\mapsto \mu_t^\tau$ is Lipschitz continuous w.r.t the Wasserstein distance with Lipschitz
constant $C_o$ which is
independent of $\tau.$
\label{thm:flow1}
\end{theorem}

\begin{proof}

{\bf Step 1. Construction of a discrete solution}

For given $\tau>0,$ we fix an integer $N$ and divide $[0,T]$ in $N$ equal subintervals of length $h=T/N.$  We build discrete solutions
$\mu_{t, \tau}^N$ satisfying:
\\
(a) The Lipschitz constant  of $t\mapsto \mu_{t, \tau}^N\in \M$ is less than $C_o.$
\\
(b) For all $t\in[0,T],$ we have $supp(\mu_{t, \tau}^N) \subset B_0(e^{\frac{(1+k)T}{\tau}}r)$ if $supp(\bar{\mu}) \subset B_0(r).$ Here, $k>0$ is same
as in (H1).
\\
(c) $\mu_{t, \tau}^N$ satisfies the discrete Hamiltonian equation
 \begin{align}
\frac{d}{dt}\mu_{t, \tau}^N + \nabla \cdot(w_{t, \tau}^N\mu_{t, \tau}^N) = 0, \qquad t \in (0 , T),
 \end{align}
with
\begin{equation}
w_{t, \tau}^N=\mathbb{J}\frac{Id - \bar{\gamma}_{\mu_{t, \tau}^N}^{\nu_{t, \tau}^N}}{\tau} \quad  for \quad t=0,h,2h,\cdots,Nh,
\end{equation}
where $ \nu_{t, \tau}^N \in J_\tau[\mu_{t, \tau}^N ]$ and $\gamma \in \Gamma_o(\mu_{t, \tau}^N,\nu_{t, \tau}^N).$

 Since $N$ and $\tau$ are fixed,  we use the notation $\mu_t:=\mu_{t,\tau}^N$ for convenience.

\noindent
(i) We build the solution in $[0,h].$

 Let us call $\mu_0:=\bar{\mu}$ and choose $\nu_0\in J_\tau[\mu_0]$ by (H1). We fix $\gamma \in \Gamma_o(\mu_0,\nu_0)$ and set

\begin{equation}\label{eq1:w0:theorem:existence}
w_0:= \mathbb{J}\frac{Id- \bar{\gamma}_{\mu_0}^{\nu_0}} {\tau}.
\end{equation}
We define

$$ \mu_t := (Id + tw_0)_{\#}
\mu_0, \qquad w_t :=
 \frac{(Id+tw_0)_{\#}(w_0\mu_0)}{\mu_t}, \qquad t\in [0,h]. $$
\\
We claim that $w_t$ is a velocity field for $\mu_t,$ that is,
\begin{equation}
\frac{d}{dt}\mu_t + \nabla \cdot (w_t\mu_t)=0,
\end{equation}
holds in the distribution sense in $(0,h)$. Indeed, for any
 $\phi \in C_c^{\infty}(\RD)$, we have

\begin{align}\label{eq1:discrete}\nonumber
\frac{d}{dt} \int_{\RD} \phi d\mu_t &= \frac{d}{dt} \int_{\RD}
\phi(Id+ tw_0) d\mu_0 \\
&= \int_{\RD} \langle\nabla \phi(x+tw_0(x)), w_0(x)\rangle d\mu_0(x) = \int_{\RD} \langle\nabla \phi, w_t\rangle d\mu_t.
\end{align}
Notice that Lemma 7.1 in \cite{AG} gives
\begin{equation}\label{eq1:theorem:existence}
\int_{\RD} |w_t|^2 d\mu_t \leq \int_{\RD} |w_0|^2 d\mu_0,
\end{equation}
\noindent
for all $t\in[0,h].$ Jensen's inequality with (H1) gives

\begin{equation}\label{HS1}
\int_{\RD} |w_0|^2 d\mu_0 = \int_{\RD}\bigl | \frac{Id-\bar{\gamma}_{\mu_0}^{\nu_0}}{\tau} \bigr |^2
d\mu_0 \leq \frac{1}{\tau^2}W_2^2(\mu_0 , \nu_0) \leq C_o^2.
\end{equation}

We exploit Proposition \ref{prac2curves} with (\ref{eq1:theorem:existence}) and (\ref{HS1}), to conclude that $t\mapsto \mu_t$ is Lipschitz
continuous with a Lipschitz constant $C_o.$

Next we show the bound on the support. Since $supp(\mu_0)\subset B_0(r),$ we have $supp(\nu_0)\subset B_0(kr)$ by (H1). Hence, if
$z\in supp(\mu_t)$ then
$$|z|\leq   \sup_{x\in supp(\mu_0)}\bigl |x + t\mathbb{J}\frac{x-\bar{\gamma}_{\mu_0}^{\nu_0}(x)}{\tau}  \bigr | \leq \bigl (r+ h\frac{r+kr}
{\tau} \bigr ),$$
hence we have $supp(\mu_t) \subset B_0 \bigl ((1+\frac{1+k}{\tau}h)r \bigr ).$
\\
(ii) We continue this process in $[h,2h].$

Since we have $W_2(\mu_0,\mu_{h})\leq hC_o\leq R_o,$ we can choose
$\nu_{h}\in J_\tau[\mu_{h}]$ and set

\begin{equation}\label{eq:w-h:theorem:existence}
 w_{h} := \mathbb{J} \frac{Id-\bar{\gamma}_{\mu_{h}}^{\nu_{h}}}{\tau}.
\end{equation}
 We introduce the following extension for $t \in (h,2h],$

\begin{equation*}
\mu_t = \bigl (Id + (t-h) w_h\bigr )_{\#} (\mu_{h}),
\qquad w_t = \frac{\bigl (Id + (t-h) w_{h} \bigr )_{\#}
(w_{h} \mu_{h})}{\mu_t}.
\end{equation*}

As in (i), we can check $t\mapsto \mu_t$ is Lipschitz continuous with a Lipschitz constant $C_o$ in $[h,2h].$ Furthermore,
the equation (\ref{eq1:discrete}) holds and we have $supp(\mu_t) \subset B_0((1+\frac{1+k}{\tau}h)^2r)$ for all $t\in [h,2h].$
\\
(iii)  We iterate the above process until we get a Lipschitz curve $t\mapsto \mu_t\in \mathcal{M}$ with Lipschitz constant $C_o.$
The curve satisfies (\ref{eq1:discrete}) for a.e $t\in(0,T)$ and hence
$$\frac{d}{dt}\mu_t + \nabla \cdot(w_t\mu_t)=0,$$
 holds in the distribution sense.
Furthermore,  for all $t\in [0,T]$
$$supp(\mu_t) \subset B_0\bigl ((1+\frac{1+k}{\tau} \frac{T}{N})^Nr\bigr ) \subset B_0(e^{\frac{(1+k)T}{\tau}}r).$$
Recalling that  $\mu_{t,\tau}^N:=\mu_t$ completes the proof.\\
\\
{\bf Step 2. Let $N$ increase to $\infty$}

From (a), $t\mapsto \mu_{t, \tau}^N$ are equi-bounded in
$\M$ and equi-Lipschitz with Lipschitz constant $C_o.$ 
Since $\mu_{t.\tau}^N $ have uniformly bounded supports, we may assume(up to a subsequence) that $\mu_{t,\tau}^N$ converges in the Wasserstein
topology as $N\rightarrow \infty.$ That is, there exist $\mu_t^\tau$ such that
$W_2(\mu_{t,\tau}^N,\mu_t^{\tau})\rightarrow 0$ for any $t\in [0,T].$
Moreover, $t\mapsto \mu_t^\tau$ is Lipschitz continuous with Lipschitz constant $C_o.$ As shown in \cite{AG},
$\mu_t^{\tau}$ solves $\frac{d}{dt}\mu_t^\tau+\nabla \cdot (w_t^\tau\mu_t^\tau)=0$  with the following property

$$w_t^{\tau} \mu_t^{\tau} \in \bigcap_{M=1}^{\infty} \bar{co}\{  w_{t,\tau}^N
\mu_{t ,\tau}^N : N\geq M \} \quad a.e \quad t\in(0,T). $$

Since
\begin{align*}
 w_{t,\tau}^N \mu_{t,\tau}^N &= \bigl(Id + (t -[Nt]/N)w_{[Nt]/N ,\tau}^N\bigr)_{\#} (w_{[Nt]/N ,\tau}^N
 \mu_{[Nt]/N ,\tau}^N) \\
&= \bigl(Id + (t -[Nt]/N)\mathbb{J} \frac{Id-\bar{\gamma}_{\mu_{[Nt]/N
,\tau}^N}^{{\nu}_{[Nt]/N ,\tau}^N}}{\tau}\bigr)_{\#} \bigl(\mathbb{J}
\frac{Id-\bar{\gamma}_{\mu_{[Nt]/N ,\tau}^N}^{{\nu}_{[Nt]/N
,\tau}^N}}{\tau} \mu_{[Nt]N,\tau}^N \bigr).
  \end{align*}
\noindent
We also obtain
\begin{equation}\label{eq2:theorem:existence}
w_t^{\tau} \mu_t^{\tau} \in \bigcap_{M=1}^{\infty} \bar{co}\bigl\{\mathbb{J} \frac{Id-
\bar{\gamma}_{\mu_{[Nt]/N ,\tau}^N}^{{\nu}_{[Nt]/N
,\tau}^N}}{\tau} \mu_{[Nt]N,\tau}^N : N \geq M \bigr\}.
\end{equation}

  Lemma \ref{MY-H2} together with (\ref{eq2:theorem:existence}) gives
 $$w_t^{\tau} \mu_t^{\tau} = \mathbb{J}v_t^{\tau}
\mu_t^{\tau},  \qquad v_t^{\tau}=\frac{Id-
\bar{\gamma}_{\mu_t^{\tau}}^{\nu_t^{\tau}}}{\tau}, $$
where $\nu_t^{\tau}\in \it{J}_\tau[\mu_t^{\tau}]$ for a.e $t\in(0,T).$  This with Lemma \ref{MYSP} concludes the proof.
\end{proof}

\subsection{Stability of Hamiltonian flows}

\begin{theorem}
Let $H : \M \rightarrow (-\infty, \infty]$ be a proper and lower semi-continuous functional satisfying (H1) and (H2). We assume that
$\bar{\mu}\in\M$ has a bounded support. For each $\tau\in (0,1),$ let $\mu^\tau$ be the solution of
the system (\ref{eq-5}) in Theorem \ref{thm:flow1}.  Then, $\{\mu^\tau\}_{\tau>0}$ (up to a sequence) converges to a solution of the Hamiltonian system
\begin{equation}\label{HF2}
   \left \{
  \begin{array}{l}
 \frac{d}{dt}\mu_t + \nabla \cdot(\mathbb{J} v_t\mu_t) = 0, \quad \mu_0 = \bar{\mu}\quad t \in (0 , T) \\
 v_t \in \partial_{-}  H(\mu _t)\cap T_{\mu_t}\M \quad a.e\quad t\in (0,T)  , \\
  \end{array}
  \right.
\end{equation}
 as $\tau$ converges to $0.$
\label{thm-hf}
\end{theorem}

\begin{proof}
Since $t\mapsto \mu_t^{\tau}$ are equi-bounded in
$\M$ and equi-Lipschitz continuous,  we may assume that, for any $t\in[0,T]$,
$\mu_t^{\tau_n}$ converges narrowly to $\mu_t$ as $\tau_n
\rightarrow 0,$ for some subsequence $\tau_n.$ 

By the same reasoning as step 2 in the proof of Theorem \ref{thm:flow1},
$\mu_t$ solves
\begin{equation}\label{eq1:theorem:stability}
   \left \{
  \begin{array}{l}
  \frac{d}{dt}\mu_t + \nabla \cdot(\mathbb{J} v_t\mu_t) = 0, \qquad t \in (0 , T) \\
   \mu_0 = \bar{\mu},\\
  \end{array}
  \right.
\end{equation}
with
\begin{equation}\label{eq2:theorem:stability}
v_t \mu_t \in \bigcap_{M=1}^{\infty} \bar{co}\{  v_t^{\tau_n}
\mu_t^{\tau_n}: n \geq M \},
\end{equation}
 for a.e $t \in (0,T).$ Here
\begin{equation}\label{eq3:theorem:stability}
 v_t^{\tau_n}=\frac{Id-
\bar{\gamma}_{\mu_t^{\tau_n}}^{\nu_t^{\tau_n}}}{\tau_n} \in \partial^+
 H(\mu_t^{\tau_n}) \cap T_{\mu_t^{\tau_n}}\M, \qquad \nu_t^{\tau_n}\in \it{J}_{\tau_n}[\mu_t^{\tau_n}].
\end{equation}
 From Lemma \ref{lemma:stability} together with (\ref{eq2:theorem:stability}) and (\ref{eq3:theorem:stability}), we know  $\nu_t^{\tau_n}\rightarrow \mu_t$
 narrowly as $\tau_n \rightarrow 0$ and
 \begin{equation}\label{eq4:theorem:stability}
 v_t \mu_t
\in \bigcap_{M=1}^{\infty} \bar{co}\{ \xi_t^{\tau_n}
\nu_t^{\tau_n}: n \geq M \},
\end{equation}
where $\nu_t^{\tau_n}\in J_{\tau_n}[\mu_t^{\tau_n}]$ and
$\xi_t^{\tau_n} = \frac{ \bar{\gamma}_{\nu_t^{\tau_n}}^{\mu_t^{\tau_n}} -Id}{\tau_n} \in
\partial_{-}H(\nu_t^{\tau_n})\cap T_{\nu_t^{\tau_n}}\M.$

By (H2) and (\ref{eq4:theorem:stability}), we get
$$v_t \in \partial_{-}H(\mu_t)\cap T_{\mu_t}\M, \qquad a.e \quad t\in(0,T), $$
which concludes the proof.

\end{proof}

\subsection{Example}

  Let $\mu_o \in \mathcal{M}$ have a bounded support and $a>0,$ we define
 \begin{equation}\label{Example:H}
 H(\mu):= -\frac{a}{2}W_2^2(\mu,\mu_o) + \int V(x) d\mu(x) + \iint W(x,y)d\mu(x)d\mu(y),
 \end{equation}
where $V:\RD \rightarrow \R$ is $\lambda_V-$convex for some $\lambda_V\in \R$, and $W:\RD \times \RD \rightarrow \R$ is convex and even.
Assume also that both are differentiable and have at most quadratic
growth at infinity. Then, the function  $H: \mathcal{M} \rightarrow  (-\infty, \infty]$ as in (\ref{Example:H})
satisfies (H1) and (H2).

\begin{proof}
First, we notice that $H$ defined as (\ref{Example:H}) is $(\lambda_V-a)-$convex and locally Lipschitz(\cite{AG}), {\it i.e.} there exist
$R_{\bar{\mu}}, C_{\bar{\mu}}>0$ such that
\begin{equation}\label{eq1:Example}
H(\mu_1)-H(\mu_2) \leq C_{\bar{\mu}}W_2(\mu_1,\mu_2),
\end{equation}
for all $\mu_i$  with $W_2(\bar{\mu},\mu_i)< R_{\bar{\mu}},$ $i=1,2.$

Secondly, $H$ satisfies the convexity condition (\ref{eq1:Remark:H1}) with $\lambda=\lambda_V-a$ (Chapter 9 in \cite{ags:book}).
From Remark \ref{Remark:H1}, it follows that for all sufficiently small $\tau>0$ and $\forall \mu \in \M,$
there exists a unique $\mu_\tau\in J_\tau[\mu].$ Furthermore, for fixed $\tau,$ $\mu\mapsto \mu_\tau\in J_\tau[\mu]$ is continuous.

We now show there is $0<R_o<R_{\bar{\mu}}$ such that for all sufficiently small $\tau>0,$
\begin{equation}\label{eq2:Example}
if \quad W_2(\mu,\bar{\mu})< R_o \quad then \quad W_2(\bar{\mu},\mu_\tau)< R_{\bar{\mu}}.
\end{equation}
Once we have (\ref{eq2:Example}),  (\ref{eq1:Example}) with $\mu_\tau \in J_\tau[\mu]$ gives
\begin{align*}
\frac{1}{2\tau}W_2^2(\mu,\mu_\tau)   \leq  H(\mu) - H(\mu_\tau)\leq C_{\mu}W_2(\mu,\mu_\tau),
\end{align*}
for all $\mu$ with $W_2(\mu,\bar{\mu})< R_o.$ That is,
\begin{equation}\label{eq3:remark}
\frac{1}{\tau}W_2(\mu,\mu_\tau)\leq 2C_{\bar{\mu}},
\end{equation}
which implies that (\ref{Condition:H1}) in (H1) holds with $C_o=2C_{\bar{\mu}}.$

Now, let us prove (\ref{eq2:Example}). We define $R_o:= R_{\bar{\mu}}/2.$ If $W_2(\mu,\bar{\mu})< R_o$ then we have
\begin{align}\label{eq4:remark}\nonumber
\frac{1}{2\tau}W_2^2(\mu,\mu_\tau) + H(\mu_\tau) &\leq  H(\mu) \\\nonumber
&\leq H(\bar{\mu}) + C_{\bar{\mu}}W_2(\bar{\mu},\mu) \\
&\leq H(\bar{\mu}) + C_{\bar{\mu}}R_o:= C.
\end{align}
We need to estimate $H(\mu_\tau)$ in (\ref{eq4:remark}).  Since $V$ and $W$ have at most quadratic growth at infinity, there are constants $c_1$ and $c_2$ such that
\begin{equation}
|V(x)| \leq c_1|x|^2 + c_2 \quad {\it and} \quad |W(x,y)|\leq c_1(|x|^2+|y|^2) +c_2,
\end{equation}
for all $x,y\in \RD$. This implies
\begin{align}\label{eq5:remark}\nonumber
H(\mu_\tau)&= -\frac{a}{2}W_2^2(\mu_\tau,\mu_o) + \int V(x) d\mu_\tau(x) + \iint W(x,y)d\mu_\tau(x)d\mu_\tau(y)\\\nonumber
& \geq -\frac{a}{2}W_2^2(\mu_\tau,\mu_o) - 3c_1\int|x|^2  d\mu_\tau(x) - 2c_2\\
& = -\frac{a}{2}W_2^2(\mu_\tau,\mu_o) - 3c_1W_2^2(\mu_\tau,\delta_0) - 2c_2.
\end{align}
We combine (\ref{eq4:remark}) and (\ref{eq5:remark}), and get
\begin{equation}\label{eq6:remark}
\frac{1}{2\tau}W_2^2(\mu,\mu_\tau) \leq \frac{a}{2}W_2^2(\mu_\tau,\mu_o) + 3c_1W_2^2(\mu_\tau,\delta_0) + 2c_2 +C.
\end{equation}
Now let $\tau \rightarrow 0$ in (\ref{eq6:remark}). Since $\mu_o, \delta_0$ are fixed and $W_2(\bar{\mu},\mu)\leq R_o= R_{\bar{\mu}}/2$,
\begin{equation}
W_2(\mu,\mu_\tau) \rightarrow 0 \quad {\it uniformly \quad  w.r.t }\quad \mu\quad as \quad \tau \rightarrow 0,
\end{equation}
which implies (\ref{eq2:Example}).

To finish proving (H1), it remains to prove (\ref{Condition-2:H1}). Let $\tilde{H}: \M \rightarrow (-\infty,\infty]$ be defined by
$$\tilde{H}(\mu)= -\frac{a}{2}W_2(\mu,\mu_o)$$
For the Hamiltonian $\tilde{H},$ it was shown that (\ref{Condition-2:H1}) holds for some $k>0$ ({\it refer} \cite{Kim}).
It is easy to see (\ref{Condition-2:H1}) holds with same $k>0$ for the Hamiltonian $H$ as in
(\ref{Example:H}).

Hence, for the Hamiltonian $H$ defined by (\ref{Example:H}), the assumption (H1) holds  with $C_o=2C_{\bar{\mu}}$ and $R_o=R_{\bar{\mu}}/2$.
It was shown in \cite{AG} that the assumption (H2) also holds.

\end{proof}

{\bf Comments:}
Suppose we want to solve the finite dimensional Hamiltonian system which consists of a single particle
\begin{equation}\label{finite}
   \left \{
  \begin{array}{l}
 x''(t)= - \nabla V(x(t)) \\

 x'(0)=\bar{v}, x(0)=\bar{x}, \\
  \end{array}
  \right.
\end{equation}
where $V:\RD\rightarrow \R$ is given and $\bar{v},\bar{x}\in \RD.$ If $V$ is not everywhere differentiable then we may try a regularization scheme as
follows. We first solve the approximate system
\begin{equation}\label{finite-approx}
   \left \{
  \begin{array}{l}
 x_{\epsilon}''(t)= - \nabla V_{\epsilon}(x_{\epsilon}(t)) \\

 x_{\epsilon}'(0)=\bar{v}, x_{\epsilon}(0)=\bar{x}, \\
  \end{array}
  \right.
\end{equation}
where $V_{\epsilon}$ is any regular approximation of $V.$ For example, we can define $V_{\epsilon}:= \rho_{\epsilon}\ast V$ as the standard
mollification of $V$. Next we check if the solution $x_{\epsilon}(t)$ of (\ref{finite-approx}) converges to a solution $x(t)$ of (\ref{finite}) as
$\epsilon$ goes to zero. Of course, we need certain properties on $V$ to ensure this stability property hold. For example,
if $V$ is convex then the limiting solution $x(t)$ satisfies the differential inclusion $ x''(t)\in -\partial_-V(x(t))$  instead of
the first equation in (\ref{finite}).

Let us address the Hamiltonian system (\ref{eq1}) in the Wasserstein space. As we saw in the previous sections, under certain
conditions on the Hamiltonian $H$, the Hamiltonian system is stable with respect to the Moreau-Yosida approximation.
Therefore, we may apply the Moreau-Yosida approximation scheme to study non locally subdifferentiable Hamiltonians.

Let $\bar{\mu}\in \M$ and $\mathcal{B}$ be a neighborhood of $\bar{\mu}$ in the Wasserstein space.
Suppose our Hamiltonian $H$ is subdifferentiable  only in a proper subset $\mathcal{D} \subset \mathcal{B},$ and $\bar{\mu}\in \mathcal{D}$.
We want to solve the system (\ref{eq1}) with the initial measure $\bar{\mu}.$ To do this, we need an algorithm
to construct solutions which stay in the subset $\mathcal{D}.$

In Theorem \ref{thm-super}, we construct {\it approximate} solutions $\nu_\tau$ for $H$ as well as solutions $\mu_\tau$ for $H_\tau.$
Notice that we have $\nu_\tau \in \mathcal{D}$ and we need only the assumption (H1) on $H$ in Theorem \ref{thm-super}. Note, the assumption (H1) has
nothing to do with the subdifferentiability of $H,$ which means that the construction of approximate solutions $\nu_\tau$
  relies entirely on the Moreau-Yosida approximation method.
 Next, in Theorem \ref{thm-hf}, we add the assumption (H2) on $H$ which then implies the convergence of $\nu_\tau $ to a solution of the system (\ref{eq1}).
The assumption (H2) does not require our Hamiltonian $H$ to be subdifferentiable everywhere in the neighborhood $\mathcal{B}$ of $\bar{\mu}.$
Instead, it requires $\mathcal{D}$ is closed in the weak* topology and (\ref{H2:AG}) hold.

Hence, as a direct result of our stability result, the Moreau-Yosida approximation
scheme provides an algorithm to construct a solution of the system (\ref{eq1}) for Hamiltonians  which are subdifferentiable only in a proper
subset $\mathcal{D}$ of a neighborhood of $\bar{\mu}.$

%
%

%

\thanks{{\bf Acknowledgements} : I would like to acknowledge my gratitude to Jacob Bedrossian and Ryan Hynd for valuable suggestions.}



\end{document}